\documentclass{proc-l}

\usepackage{amssymb}

\newcommand{\ds}{\displaystyle}
\newcommand{\R}{{\mathbb{R}}}
\newcommand{\bxi}{{\boldsymbol\xi}}
\newcommand{\ttt}{{\boldsymbol{\tau}}}
\newcommand{\e}{{\boldsymbol{e}}}
\newcommand{\h}{{\boldsymbol{h}}}
\renewcommand{\k}{{\boldsymbol{k}}}
\renewcommand{\u}{{\boldsymbol{u}}}
\renewcommand{\a}{{\boldsymbol{a}}}
\newcommand{\x}{{\boldsymbol{x}}}
\newcommand{\y}{{\boldsymbol{y}}}
\newcommand{\w}{{\boldsymbol{w}}}
\renewcommand{\v}{{\boldsymbol{v}}}
\newcommand{\vi}{{\boldsymbol{v}_i}}
\newcommand{\ed}{{\boldsymbol{e}_d}}
\newcommand{\ei}{{\boldsymbol{e}_i}}
\newcommand{\ej}{{\boldsymbol{e}_j}}
\newcommand{\pp}{{\mathcal{P}}}
\newcommand{\prd}{\pp_{r,d}}
\renewcommand{\pod}{\pp_{1,d}}
\newcommand{\pro}{\pp_{r,1}}
\newcommand{\poo}{\pp_{1,1}}
\newcommand{\eps}{\varepsilon}
\def\be{\begin{equation}}
\def\ee{\end{equation}}

\newtheorem{theorem}{Theorem}[section]
\newtheorem{lemma}[theorem]{Lemma}
\newtheorem{corollary}[theorem]{Corollary}

\theoremstyle{definition}

\theoremstyle{remark}
\newtheorem{remark}[theorem]{Remark}

\numberwithin{equation}{section}

\begin{document}

\title{Discrete $d$-dimensional moduli of smoothness}

\author{Z. Ditzian}
\address{Department of Mathematical and Statistical Sciences, University of Alberta, Edmonton, AB, T6G2G1, Canada}
\email{zditzian@math.ualberta.ca}

\author{A. Prymak}
\address{Department of Mathematics, University of Manitoba, Winnipeg, MB, R3T2N2, Canada}
\email{prymak@gmail.com}
\thanks{The second author was supported by NSERC of Canada.}



\date{}

\commby{Walter Van Assche}

\begin{abstract}
We show that on the $d$-dimensional cube $I^d\equiv [0,1]^d$ the discrete moduli of smoothness which use only the values of the function on a diadic mesh are sufficient to determine the moduli of smoothness of that function. As an important special case our result implies for $f\in C(I^d)$ and given integer $r$ that when $0<\alpha<r$, the condition
\[
\Bigl|\Delta^r_{2^{-n}\ei}f\bigl(\frac{k_1}{2^n},\dots,\frac{k_d}{2^n}\bigr)\Bigr|\le M2^{-n\alpha}
\]
for integers $1\le i\le d$, $0\le k_i\le 2^n-r$, $0\le k_j\le 2^n$  when $j\ne i$, and $n=1,2,\dots$\\
is equivalent to
\[
\Bigl|\Delta^r_{h\e}f(\bxi)\Bigr|\le M_1 h^\alpha\quad\textrm{for }\bxi,\e\in\R^d,\ h>0
\textrm{ and }|\e|=1\textrm{ such that }\bxi,\bxi+rh\e\in I^d.
\]
\end{abstract}

\maketitle

\section{Introduction}
The $r$-th modulus of smoothness on $C(I^d)$, the space of functions continuous on $I^d$, $\omega^r(f,t)$ is given by
\be\label{1.1}
\omega^r(f,t)=\sup_{|\h|\le t}\max_{\x,\x+r\h\in I^d}\bigl|\Delta^r_\h f(\x)\bigr|,
\ee
where
\be
\Delta^r_\h f(\x)\equiv \sum_{k=0}^r(-1)^{r-k} {r\choose k} f(\x+k\h), \quad \Delta_\h f(\x)\equiv \Delta^1_\h f(\x),
\ee
and where ``$\equiv$'' stands for ``by definition''. Clearly, it is desirable to use information on $f(\x)$ at fewer points $\x$. The moduli of smoothness of a function play an important role in the investigation of the rate at which various approximation processes converge. Moreover, the data available is often of a discrete nature, and hence, we believe that it will be helpful to measure smoothness using only such data. In this paper it will be shown that the values of $f$ at the points $(\frac{k_1}{2^n},\frac{k_2}{2^n},\dots,\frac{k_d}{2^n})$ are sufficient to give us information on $\omega^r(f,t)$. Our main result is:
\begin{theorem}\label{thm1.1}
Suppose $f\in C(I^d)$ and $\k\equiv(k_1,\dots,k_d)$
\be\label{1.3}
\sup_{i,\k}\Bigl\{\bigl|
\Delta^r_{2^{-n}\ei}f(2^{-n}\k)
\bigr|:
2^{-n}\k,2^{-n}(\k+r\ei)\in I^d\Bigr\}
\equiv\Psi_r(f,n)\equiv\Psi_r(n),
\ee
where $\ei$ is the unit vector in the $i$-th direction, and $\k$ has integer components.
Then
\be\label{1.4}
\bigl|\Delta^r_{t\ei}f(\u)\bigr|\le M_1(r,d) \sum_{k=0}^\infty \Psi_r(n+k),
\quad\textrm{for }\u,\u+rt\ei\in I^d,
\textrm{ and }t\le 2^{-n},
\ee
\be\label{1.5}
\omega^r(f,t)\le M_2(r,d) \Bigl(\sum_{k=0}^\infty\Psi_r(n+k)+\sum_{k=1}^{n_0}2^{kr}\Psi_r(n-k)+t^r\|f\|\Bigr)
\textrm{ for }2^{-n-1}<t\le 2^{-n},
\ee
where $n_0$ is the largest integer satisfying $r2^{n_0-n-1}\le1$, and
\be\label{1.55}
\omega^1(f,t)\equiv\omega(f,t)\le M \sum_{k=0}^\infty\Psi_1(n+k)
\quad\textrm{for }t\le 2^{-n},
\ee
where $M_1(r,d)$, $M_2(r,d)$ and $M$ are independent of $n$ and $f$.
\end{theorem}

The discrete $r$-th modulus of smoothness is $\Psi_r(f,n)$ defined in~\eqref{1.3}. Clearly, $\Psi_r(f,n)$ depends only on the values of $f$ at the diadic mesh $2^{-n}\k$.

Theorem~\ref{thm1.1} implies that in the direction of the axes we have a somewhat better estimate, and using appropriate references, it will be shown that this is inherent to this problem and not a result of weakness in the proof. We also show that $\Psi_r(f,n)=o(2^{-nr})$ implies that $f$ is a polynomial of degree $\le r-1$ in each variable, which is the small $o$ saturation class. For $d=1$ a somewhat more general result than Theorem~\ref{thm1.1} (in some ways) was proved in~\cite[p.~119]{Di87}. In addition, for $d=2$ and $r=2$ a somewhat weaker result was proved in~\cite[p.~314]{Di88}. As part of the proof we will prove again the result for $d=1$ as the present different construction will be needed for the general case. The process we use consists of the construction of a sequence of spline functions which are locally in $\prd$, where $\prd$ is the set of polynomials of degree smaller than $r$ in each variable. That sequence will converge to our function $f$ and will yield our result.

We were fortunate that a result on determinants which we needed was proved recently (in 1999) by Ratlieff and Rush (see~\cite{RaRu}).

We prove Theorem~\ref{thm1.1} in \S4 using lemmas for $d=1$ and for $d>1$ established in \S2 and \S3 respectively. In \S5 we present several remarks and corollaries that will make Theorem~\ref{thm1.1} easier to apply, demonstrate the need for its differen parts, and prove the small $o$ saturation result.

\section{The crucial lemma for $d=1$}
In this section we give the lemma that will be crucial in the proof of Theorem~\ref{thm1.1} for $d=1$ and afterwards for other $d$. In fact, this lemma is the heart of the matter.

\begin{lemma}\label{lm2.1}
Suppose $a\in\R$, $r\ge2$, $h>0$, $g\in C[a,a+2(r-1)h]$ and that
\[
g|_{[a,a+(r-1)h]}\in\pro,\quad
g|_{[a+(r-1)h,a+2(r-1)h]}\in\pro
\]
and suppose also that $g(a+2kh)=0$ for $k=0,\dots,r-1$. Then
\be
\|g\|_{C[a,a+2(r-1)h]}\le c(r)\max_{j=0,\dots,r-2}\bigl|\Delta^r_h g(a+jh)\bigr|
\ee
where $c(r)$ is independent of $a$, $h$ and $g$.
\end{lemma}
\begin{proof}
We set for $1\le j\le r-1$ $g(a+(2j-1)h)=v_j$, and as $g|_{[a,a+(r-1)h]}$ and $g|_{[a+(r-1)h,a+2(r-1)h]}$ are polynomials of degree smaller than $r$ (that is, they are in $\pro$), they can be constructed as the Lagrange interpolation of $g$ using the points $a+ih$ where $0\le i\le r-1$ and  the points $a+ih$ where $r-1\le i\le 2(r-1)$ respectively. More precisely,
\[
g|_{[a,a+(r-1)h]}=\sum_{i=0}^{r-1} g(a+ih) l_i, \quad\text{where}\quad l_i(x)=\prod_{\substack{j=0\\j\ne i}}^{r-1}\frac{x-(a+jh)}{(a+ih)-(a+jh)}
\]
are the Lagrange basis polynomials. Clearly, $\ds\max_{x\in[a,a+(r-1)h]}|l_i(x)|\le c_1(r)$ with $c_1(r)$ independent of $a$ and $h$. The same bound is valid for the Lagrange basis polynomials of $g|_{[a+(r-1)h,a+2(r-1)h]}$. Therefore,
\[
\|g\|_{C[a,a+2(r-1)h]}\le c_2(r) \max_{0\le k\le 2(r-1)} |g(a+kh)|=c_2(r)\max_{1\le j\le r-1}|v_j|.
\]
We will complete the proof when we show
\be\label{2.2}
\max_{1\le j\le r-1}|v_j|\le c_3(r)\max_{0\le i\le r-2}\bigl|\Delta^r_hg(a+ih)\bigr|.
\ee
While for small $r$ one can easily verify~\eqref{2.2} directly, it gets complicated for higher $r$. For $i=1,\dots,r-1$ we can write
\begin{align*}
\Delta^r_hg(a+(i-1)h) &= \sum_{k=0}^r (-1)^{r-k} {r \choose k} g(a+(i-1+k)h) \\
&= \sum_{i\le 2j\le r+i}(-1)^{r+i-2j}{r \choose 2j-i} g(a+(2j-1)h) \\
&= (-1)^{r+i}\sum_{j=1}^{r-1}{r \choose 2j-i}v_j
\end{align*}
where ${r \choose 2j-i}\equiv0$ for $2j-i\not\in[0,r]$. Essentially, this is a linear transformation from the vector $(0,v_1,0,v_2,\dots,v_{r-1},0)$ to the vector $(\Delta^r_hg(a),\dots,\Delta^r_hg(a+(r-2)h))$. The transformation is represented by the $(r-1)\times(2r-1)$ matrix whose entries are the binomial coefficients, $a_{m,l}=(-1)^{r-l+m}{r \choose l-m}$ for $1\le m\le r-1$, $1\le l\le2r-1$ when $0\le l-m\le r$ and $0$ elsewhere. As the vector $(0,v_1,0,v_2,\dots,v_{r-1},0)$ consists of zeros in its odd entries, we may consider the transformation from $(v_1,\dots,v_{r-1})$ into $(\Delta^r_hg(a),\dots,\Delta^r_hg(a+(r-2)h))$ which is the $(r-1)\times(r-1)$ matrix of the binomial coefficients $a_{m,l}$ from which all the odd columns were eliminated. (This is in fact $\Delta^r_hg(a+(i-1)h)=(-1)^{r+i}\sum_{j=1}^{r-i}{r \choose 2j-i}v_j$ where ${r \choose 2j-i}=0$ for $2j-i\not\in[0,r]$.) It was recently proved by Ratlief and Rush (see~\cite[Theorem~1.1]{RaRu}) that the absolute value of the determinant of such a matrix is not zero and is in fact equal to $2^{r(r-1)/2}$. Therefore, this matrix has an inverse and~\eqref{2.2} is proved (with $c_2(r)$ depending only on the matrix).
\end{proof}

The analogue for $r=1$ is essentially trivial and is given by the following result that will serve to prove Theorem~\ref{thm1.1} for $r=1$, first for $d=1$ and then for other $d$.
\begin{lemma}\label{lm2.2}
Suppose $a\in\R$, $g\in L_\infty[a,a+2h]$,
\[
g|_{[a,a+h)}\in\poo, \quad g|_{[a+h,a+2h)}\in\poo,
\]
and $g(a)=0$. Then $\|g\|_{L_\infty[a,a+2h)}=|\Delta_hg(a)|$.
\end{lemma}
\begin{proof}
Observe that $\poo$ is the set of constants on $\R$.
\end{proof}

\section{The crucial lemma for $d>1$}
For the proof of Theorem~\ref{thm1.1} for $d>1$ we need to extend Lemma~\ref{lm2.1} (and Lemma~\ref{lm2.2}) to higher dimensions. That is, we will prove the following result.
\begin{lemma}\label{lm3.1}
Suppose $\a\in\R^d$, $r\ge2$, $g\in C(\a+2(r-1)hI^d)$ for some $h>0$, $g|_{\a+\v (r-1)h+(r-1)hI^d}\in\prd$ where $\v\in\R^d$ is any vector whose entries are $1$ or $0$, and suppose $g(\a+2\v kh)=0$ for $1\le k\le r-1$ and any $\v$ as above. Then
\be
\|g\|_{C(\a+2(r-1)hI^d)}\le c(r,d) \max_{i,\v_i}\bigl|\Delta^r_{h\ei}g(\a+\vi h)\bigr|
\ee
where $\vi\in\R^d$ given by $\vi=(v_{i,1},\dots,v_{i,d})$ and $v_{i,j}$ can take the values $0,\dots,2(r-1)$ when $i\ne j$, while $v_{i,i}$ can take the values $0,\dots,r-2$.
\end{lemma}

We note that $\a+2(r-1)hI^d$ is a cube of side $2h(r-1)$ parallel to the axes. This cube is divided into $2^d$ cubes $\a+\v (r-1)h+(r-1)hI^d$ on which the function $g$ is a polynomial of degree $<r$ in the directions of the axes. We also note that $g$ equals $0$ on the $r^d$ points $\a+2\v kh$, $1\le k\le r-1$, of the mesh of points which are spaced at equal distances of $2h$ in each direction.

For $r=1$ we define $\tilde I^d=[0,1)^d$ and the needed lemma whose proof is trivial is given as follows:
\begin{lemma}\label{lm3.2}
Suppose $\a\in\R^d$, $g\in L_\infty(\a+2h\tilde I^d)$, $g|_{\a+\v h+h\tilde I^d}\in\pod$ with $\v$ of Lemma~\ref{lm3.1}, and that $g(\a)=0$. Then
\[
\|g\|_{L_\infty(\a+h\tilde I^d)}\le d\max_i\bigl( |\Delta_{h\ei}g(\a)|, \max_{j\ne i} |\Delta_{h\ei}g(\a+h\ej)| \bigr).
\]
\end{lemma}

{\em Proof of Lemma~\ref{lm3.1}.}
Using translation and dilation, which do not change the values of $g$, we may assume that $g$ is defined on $[0,2(r-1)]^d$ and that $g(\u)=0$ on the vectors $\u$ with all entries being even integers. On each of the $2^d$ cubes $(r-1)\v+(r-1)I^d$ (with entries of $\v$ being either $0$ or $1$) the $r^d$ points of that cube whose entries are integers completely determines the polynomials in $\prd$ as there are exactly $r^d$ monomials. Similarly to the proof of Lemma~\ref{lm2.1}, each of these polynomials can be written as the appropriate Lagrange interpolation, using the tensor products of the corresponding univariate Lagrange basis polynomials. Therefore, $\|g\|_{C[0,2(r-1)]^d}$ is bounded by $\max|g(\u)|$ with $\u=(u_1,\dots,u_d)$, such that $u_j=0,1,\dots,2(r-2)$, and at least one $u_j$ is an odd integer. We prove our lemma by induction. For $d=1$ it is Lemma~\ref{lm2.1}. Assume our lemma for $d-1$, and hence it is valid on the $d-1$ dimensional cubes of points in $[0,2(r-1)]^d$ with the last coordinate being even integer. That is, for $l=0,1,\dots,r-1$
\be\label{3.2}
\|g(\cdot,\dots,\cdot,2l)\|_{C(2(r-1)I^{d-1})}\le c(r,d-1)
\max_{1\le i<d}\bigl|\Delta_{\ei}^rg(u_1,\dots,u_{d-1},2l)\bigr|,
\ee
where the integers $u_k$ satisfy $0\le u_j\le 2(r-1)$ for $j\ne i$, $0\le j\le d-1$ and $u_i$ satisfies $0\le u_i\le r-2$.  The function $g(v_1,\dots,v_{d-1},x_d)$ for fixed $v_i$ integers is continuous when $x_d\in[0,2(r-1)]$ and is a polynomial of degree $\le r-1$ for $x_d\in[0,r-1]$ and for $x_d\in[r-1,2(r-1)]$. Examining $(v_1,\dots,v_{d-1})$, we encounter two situations: either all $v_i$ are even or not. In the first situation $g(v_1,\dots,v_{d-1},x_d)$ satisfies the conditions of Lemma~\ref{lm2.1} for $x_d$ and hence,
\[
|g(v_1,\dots,v_{d-1},x_d)|\le c(r,1) \max_{v_d=0,\dots,r-2}\bigl|\Delta^r_\ed g(v_1,\dots,v_{d-1},v_d)\bigr|.
\]
In the second situation we use a polynomial in $x_d$ of degree $\le r-1$ that interpolates $g(v_1,\dots,v_{d-1},x_d)$ at the points $x_d=0$, $x_d=2$, $\dots$, $x_d=2(r-1)$ which we call $P(x_d)$. The function of one variable $g(v_1,\dots,v_{d-1},x_d)-P(x_d)$ satisfies the condition of Lemma~\ref{lm2.1} for $x_d$, and hence,
\[
\|g(v_1,\dots,v_{d-1},\cdot)-P(\cdot)\|_{C[0,2(r-1)]}
\le c(r,1)
\max_{0\le v_d\le r-2}
\bigl|\Delta^r_\ed(g(v_1,\dots,v_{d-1},v_d)-P(v_d))\bigr|.
\]
Clearly, $\Delta^r_\ed P(x)=0$ and thus we only have to estimate $\|P\|_{C[0,2(r-1)]}$. Since $P\in\pro$ interpolates $g$ at  $x_d=0$, $x_d=2$, $\dots$, $x_d=2(r-1)$, we have
\[
\|P\|_{C[0,2(r-1)]}\le c(r) \max_{l=0,\dots,r-1}|P(2l)|=c(r) \max_{l=0,\dots,r-1}|g(v_1,\dots,v_{d-1},2l)|.
\]
The last quantity is estimated using the induction hypothesis~\eqref{3.2}, which completes the proof of our lemma.
\qquad\endproof

\section{Proof of the main result}
In this section we prove Theorem~\ref{thm1.1} which is the main result of this paper.

\begin{proof}[Proof of Theorem~\ref{thm1.1}.]
We first prove~\eqref{1.4}. 
We prove~\eqref{1.4} for $t\le 2^{-n-1}$ (rather than for $t\le 2^{-n}$) which may just contribute a somewhat bigger constant as
\[
|\Delta^r_{2t\ei}f(\u)|\le 2^r \max_{\v,\v+rt\ei\in I^d} |\Delta^r_{t\ei}f(\v)|.
\]
We further assume that $r-1\le2^{n-1}$.

We prove our theorem for fixed $\u$, $t$ and $i$ satisfying $\u,\u+rt\ei\in I^d$.

For $r\ge2$ we construct the basic cube $\k2^{-n}+2(r-1)2^{-n}I^d\equiv A\subset I^d$ such that $\u,\u+rt\ei\in A$. The choice of $\k=(k_1,\dots,k_d)$ where $k_j$ can take values $0,1,2,\dots$ can be made by choosing $k_j$ such that $0\le k_j 2^{-n}\le u_j\le k_j 2^{-n}+(r-1)2^{-n+1}\le 1$. To choose $k_j$ if $u_j\le\frac12$, we select $k_j=\min([2^nu_j],2^n-2(r-1))$ (where $[y]$ is the biggest integer smaller than or equal to $y$). When $u_j\ge\frac12$, we set $u_j+rt=\tilde u_j$ and $\tilde k_j=\min([2^n(1-\tilde u_j)],2^n-2(r-1))$ and define $k_j$ by $k_j=2^n-\tilde k_j-2(r-1)$.

We now construct a sequence of spline functions on $A$. $S_n$ is the polynomial of degree $\le r-1$ in the direction of the axes on $A$ interpolating $f$ at the points $\k 2^{-n}+\w 2^{-n+1}$ where $\w=(w_1,\dots,w_d)$ and $w_j$ takes values $0,\dots,r-1$. We divide $A$ into $2^r$ cubes by dividing each side by $2$, that is, the cubes $\k 2^{-n}+\v (r-1)2^{-n}+(r-1)2^{-n}I^d$, where $\v$ has entries $0$ or $1$. Then $S_{n+1}$ is a polynomial in $\prd$ in each of the cubes interpolating $f$ at $\k 2^{-n} +\v (r-1) 2^{-n}+ \w 2^{-n}$ where $w_j$ takes values $0,\dots,r-1$. By a process of dividing each cube into $2^r$ cubes and having a polynomial in $\prd$ in each cube, we obtain in the $k$-th step the spline $S_{n+k}$ defined on $A$. For $S_{n+k}$ $A$ is divided into $2^{rk}$ cubes with sides of size $(r-1)2^{-n-k+1}$ and in each cube $S_{n+k}$ is in $\prd$ interpolating $f$ at the points $\k 2^{-n} + \v (r-1) 2^{-n-k+1} + \w 2^{-n-k+1}$ (where $\v$ is a vector of integers and $w_j=0,\dots,r-1$). On each such cube $B$, $S_{n+k}=S_{n+k}(f)$ is a projection operator from $C(B)$ to $\prd$ whose norm is bounded and depends on $r$ and $d$ only. Since $\prd$ contains constant functions, we obtain
\[
\|f-S_{n+k}\|_{C(B)}\le c_1(r,d) \max_{\x,\y\in B}|f(\x)-f(\y)|,
\]
and using uniform continuity of $f$ on $A$, we conclude that the constructed sequence of splines $S_m$ converges to $f$ in $C(A)$. Hence,
\[
f(\x)=S_n(\x)+\sum_{k=1}^\infty(S_{n+k}(\x)-S_{n+k-1}(\x)).
\]
Therefore,
\[
|\Delta^r_{t\ei}f(\u)|\le |\Delta^r_{t\ei} S_h(\u)|+2^r\sum_{k=1}^\infty \|S_{n+k}-S_{n+k-1}\|_{C(A)}.
\]
As $\Delta^r_{t\ei}S_n(\u)=0$ (being in $\prd$), we have to show only that
\be\label{4.1}
\|S_{n+k}-S_{n+k-1}\|_{C(A)}\le c_2(r,d)\Psi_r(n+k-1).
\ee
In each cube of the $2^{r(k-1)}$ cubes (whose union is $A$) defining $S_{n+k-1}$, the difference $S_{n+k}-S_{n+k-1}$ satisfies the conditions about $g$ of Lemma~\ref{lm3.1} with $h=2^{-n-k+1}$ at points $\k 2^{-n}+\v (r-1) 2^{-n-k+1} + \w 2^{-n-k+1}$ in $A$. Therefore, using Lemma~\ref{lm3.1}, we have~\eqref{4.1}.

For $r=1$ we start with the cube $\k 2^{-n}+2^{-n+1}\tilde I^d\equiv \tilde A\subset I^d$ and then define $S_n$ on $\tilde A$ as the constant (element of $\pod$) of the value of $f(\k 2^{-n})$, that is the value of $f(\x)$ at the point where all coordinates of the cube are smallest. Divide $\tilde A$ into $2^r$ cubes of the same nature and define $S_{n+1}$ as the corresponding constant in each cube and so on. Using Lemma~\ref{lm3.2}, $|S_{n+k}(\x)-S_{n+k-1}(\x)|\le\Psi_1(n+k)$. We observe that $S_n(\u)-S_n(\u+\ttt)=0$ on $\tilde A$ for $\ttt=(\tau_1,\dots,\tau_d)$ where $0<\tau_i<2^{-n}$, and as $\|S_m-f\|_{L_\infty(\tilde A)}\to0$, \eqref{1.4} is valid for $r=1$. In fact, for $r=1$ we have the better estimate~\eqref{1.55}.

To prove~\eqref{1.5} we use the result from~\cite[p.~617, (4.2)]{Di84} in the following equivalent form:
\[
\omega^r(f,t)\le c\Bigl(t^r\|f\|+\sum_{l=0}^{n_0}2^{-rk}\omega^r_{\ei}(f,2^lt)\Bigr)
\]
where
\[
\omega^r_\ei(f,u)=\mathop{\sup_{\x,\x+rh\ei\in I}}_{0<h\le u}|\Delta^r_{h\ei}f(\x)|.
\]
For $2^{-n-1}<t\le 2^{-n}$ the sum stops at $r2^{l-n-1}>1$. Therefore, for $2^{-n-1}<t\le 2^{-n}$
\begin{align*}
\omega^r(f,t) &\le  c \Bigl(t^r\|f\|+\sum_{l=0}^{n_0}2^{-rl}\sum_{k=0}^\infty\Psi_r(n-l+k)\Bigr) \\
&\le  c_1 \Bigl(t^r\|f\|+\sum_{k=0}^\infty \Psi_r(n+k) + \sum_{l=0}^{n_0}2^{-rl}\Psi_r(n-l)\Bigr).\qedhere
\end{align*}
\end{proof}

\section{Remarks and corollaries}
In this section we make some additional remarks about and conclusions of the result of our paper. First we obtain the saturation result.
\begin{theorem}
Suppose $f\in C(I^d)$ and $\Psi_r(n)=o(2^{-nr})$, $n\to\infty$ for $\Psi_r(n)$ of~\eqref{1.3}. Then $\Psi_r(n)=0$ and $f\in\prd$.
\end{theorem}
\begin{proof}
If $|\Delta^r_{2^{-n}\ei}f(2^{-n}\k)|\le\eps_n 2^{-nr}$, $\eps_n\to0$, then for every $\eps>0$ and a fixed $m$ there is an integer $n$, $n>m$, such that $|\Delta^r_{2^{-n}\ei}f(2^{-n}\k)|\le\eps 2^{-nr}$ for all $\k$. Therefore, $|\Delta^r_{2^{-m}\ei}f(2^{-m}\k)|\le\eps 2^{-nr}2^{(n-m)r}=\eps 2^{-mr}$. As $\eps>0$ is arbitrary, $\Psi_r(m)=0$. The inequality~\eqref{1.4} completes the proof.
\end{proof}

\begin{remark}
For $f\in\prd$, $\Psi_r(n)=0$ but $\omega^r(f,t)$ may behave like $t^r$ times the $r$-th derivative in the direction $\e$ which is different from any $\ei$. This shows that the term $t^r\|f\|$ in~\eqref{1.5} is not redundant.
\end{remark}

As corollaries of Theorem~\ref{thm1.1}, we also have:

\begin{corollary}
If $\Psi_r(n)=O(2^{-n\alpha})$ for $0<\alpha\le r$, then~\eqref{1.4} implies $|\Delta^r_{t\ei} f(\x)|=O(t^\alpha)$. If $\Psi_r(n)=O(2^{-n\alpha})$ for $0<\alpha<r$, \eqref{1.5} implies $\omega^r(f,t)=O(t^\alpha)$. However, for $\alpha=r$, $d>1$ and $r>1$, examples in~\cite[p.~620]{Di84} show that $\omega^r(f,t)$ may behave like $t^r|\log t|$ and hence the extra terms in~\eqref{1.5} are not redundant. For $r=1$, $\Psi_r(n)=O(2^{-n\alpha})$ with $0<\alpha\le1$ implies $\omega(f,t)=O(t^\alpha)$ using~\eqref{1.55}.
\end{corollary}

\begin{corollary}
If $\Psi_r(n+1)<\lambda\Psi_r(n)$ for some $\lambda<1$ and all $n$, then~\eqref{1.4} implies $|\Delta^r_{t\ei}f(\x)|\le c\Psi_r(n)$ when $2^{-n-1}\le t\le 2^{-n}$. If in addition to the above, we have $\Psi_r(n-1)\le \mu\Psi(n)$ with $\mu<2^r$, then using~\eqref{1.4} and~\eqref{1.5}, $\omega^r(f,t)\le c\Psi_r(n)$ when $2^{-n-1}\le t\le 2^{-n}$.
\end{corollary}

\begin{remark}
One can replace $I^d$ by $\R^d$ or $\R_+^d$ in~\eqref{1.1} and in Theorem~\ref{thm1.1}, with an almost identical proof. The only difference is that the choice of the basic cube $A$ in the proof of Theorem~\ref{thm1.1} is even easier, as one can simply take $k_j=[2^n u_j]$ to ensure that $A\equiv \k2^{-n}+2(r-1)2^{-n}I^d$ contains both $\u$ and $\u+rt\ei$, while $A\subset \R^d$ or $A\subset\R_+^d$ respectively. The rest of the proof is the same, as it is concerned only with $A$.
\end{remark}

\begin{remark}
One cannot expect to derive bounds on moduli of smoothness in $L_p$ from the values on diadic mesh only, as the measure of all the points of the mesh is zero. One will need to impose severe additional conditions on the function, or to use different data related to the function, perhaps averages over small cubes.
\end{remark}

\providecommand{\bysame}{\leavevmode\hbox to3em{\hrulefill}\thinspace}
\providecommand{\MR}{\relax\ifhmode\unskip\space\fi MR }
\providecommand{\MRhref}[2]{%
  \href{http://www.ams.org/mathscinet-getitem?mr=#1}{#2}
}
\providecommand{\href}[2]{#2}

\end{document}